\documentclass[11pt,a4paper]{amsart}
\usepackage{amsmath}
\usepackage{amssymb}
\usepackage{amsthm}
\usepackage{latexsym}
\usepackage{pstricks}
\usepackage{amsbsy}
\usepackage[all]{xypic}
\usepackage{amscd}
\usepackage{mathrsfs}
\usepackage{tipa}
\usepackage{txfonts}
\usepackage{amsfonts}
\usepackage{graphicx}
\usepackage{listings}
\usepackage{url}
\usepackage{bookmark}
\newtheorem{theorem}{Theorem}[section]
\newtheorem{lemma}[theorem]{Lemma}

\newtheorem{proposition}[theorem]{Proposition}

\begin{document}

\providecommand{\ann}{\mathop{\rm ann}\nolimits}%
\providecommand{\gld}{\mathop{\rm gl. dim}\nolimits}%
\providecommand{\gord}{\mathop{\rm gor. dim}\nolimits}%
\providecommand{\pd}{\mathop{\rm pd}\nolimits}%
\providecommand{\rk}{\mathop{\rm rk}\nolimits}%
\providecommand{\Fac}{\mathop{\rm Fac}\nolimits}%
\providecommand{\ind}{\mathop{\rm ind}\nolimits}%
\providecommand{\Sub}{\mathop{\rm Sub}\nolimits}%
\providecommand{\coker}{\mathop{\rm coker}\nolimits}%
\providecommand{\cone}{\mathop{\rm cone}\nolimits}%
\def\s{\stackrel}
\def\A{\mathcal{A}}
\def\C{\mathcal{C}}
\def\D{\mathcal{D}}
\def\DA{{D^b(A)}}
\def\K{{K^b(\proj A)}}
\def\H{\mathcal{H}}
\def\T{\mathcal{T}}
\def\P{\mathcal{P}}
\def\X{\mathcal{X}}
\def\Y{\mathcal{Y}}
\def\Z{\mathcal{Z}}
\def\F{\mathcal{F}}
\def\R{\mathcal{R}}
\def\L{\mathcal{L}}
\def\CP{\C(\p)}
\def\op{\text{op}}
\providecommand{\add}{\mathop{\rm add}\nolimits}%
\providecommand{\ann}{\mathop{\rm ann}\nolimits}%
\providecommand{\End}{\mathop{\rm End}\nolimits}%
\providecommand{\Ext}{\mathop{\rm Ext}\nolimits}%
\providecommand{\Hom}{\mathop{\rm Hom}\nolimits}%
\providecommand{\inj}{\mathop{\rm inj}\nolimits}%
\providecommand{\proj}{\mathop{\rm proj}\nolimits}%
\providecommand{\rad}{\mathop{\rm rad}\nolimits}%
\providecommand{\soc}{\mathop{\rm soc}\nolimits}%
\providecommand{\thick}{\mathop{\rm thick}\nolimits}%
\providecommand{\Tr}{\mathop{\rm Tr}\nolimits}%
\renewcommand{\dim}{\mathop{\rm dim}\nolimits}%
\renewcommand{\Im}{\mathop{\rm Im}\nolimits}%
\renewcommand{\mod}{\mathop{\rm mod}\nolimits}%
\renewcommand{\ker}{\mathop{\rm ker}\nolimits}%
\renewcommand{\rad}{\mathop{\rm rad}\nolimits}%
\def \text{\mbox}
\def\ends{\end{enumerate}}
\newcommand{\id}{\operatorname{id}}
\renewcommand{\k}{\mathbf{k}}
\newcommand{\p}{\mathbf{P}}
\newcommand{\q}{\mathbf{Q}}
\newcommand{\rr}{\mathbf{R}}
\newcommand{\x}{\mathbf{X}}
\newcommand{\y}{\mathbf{Y}}
\newcommand{\z}{\mathbf{Z}}
\newcommand{\oo}{\mathbf{O}}
\newcommand{\Ll}{\mathbf{L}}
\newcommand{\m}{\mathbf{M}}
\newcommand{\n}{\mathbf{N}}
\renewcommand{\d}{\mathbf{D}}
\newcommand{\e}{\mathbf{E}}
\renewcommand{\t}{\mathbf{T}}

\title{Endomorphism algebras of 2-term silting complexes}

\author[Buan]{Aslak Bakke Buan}
\address{
Department of Mathematical Sciences
Norwegian University of Science and Technology
7491 Trondheim
NORWAY
}
\email{aslakb@math.ntnu.no}

\author[Zhou]{Yu Zhou}
\address{
Department of Mathematical Sciences
Norwegian University of Science and Technology
7491 Trondheim
NORWAY
}
\email{yu.zhou@math.ntnu.no}

\begin{abstract}
We study possible values of the global dimension of endomorphism algebras of 2-term silting complexes. We show that for any algebra $A$ whose global dimension $\gld A\leq 2$ and any 2-term silting complex $\p$ in the bounded derived category $\DA$ of $A$, the global dimension of $\End_\DA(\p)$ is at most 7. We also show that for each $n>2$, there is an algebra $A$ with $\gld A=n$ such that $\DA$ admits a 2-term silting complex $\p$ with $\gld \End_\DA(\p)$ infinite.
\end{abstract}

\thanks{
This work was supported by FRINAT grant number 231000, from the
Norwegian Research Council. Support by the Institut Mittag-Leffler (Djursholm, Sweden) is gratefully
acknowledged.
}

\maketitle

\section*{Introduction}

Let $A$ be a finite dimensional algebra over a field $k$. Let
$T$ be a (classical) tilting module in the category $\mod A$ of finite
dimensional right $A$-modules; that is the projective dimension $\pd T$
is at most $1$, we have $\Ext_A^1(T,T) = 0$ and
there is an exact sequence $0 \to A \to T_1 \to T_2 \to 0$
with $T_1,T_2$ in $\add T$, the additive closure of $T$.
Let $B = \End_A(T)$.
Then, it is a well-known fact (see for example \cite[III, Section 3.4]{ha} for a more general statement) that $\gld B \leq \gld A + 1$, where $\gld A$ denotes the global dimension of $A$.

In this paper we investigate to which extent this generalizes
to the following setting. We now consider a 2-term silting
complex $\p$ in the bounded homotopy category of finitely generated
projective $A$-modules, $K^b(\proj A)$.
This is just a map between projective $A$-modules, considered as a complex, with
the property that $\Hom_{K^b(\proj A)}(\p,\p[1]) = 0$ where $[1]$ denotes the shift functor, and such that $\p$
generates $K^b(\proj A)$ as a triangulated category. Note that  $K^b(\proj A)$ can be considered to be a full triangulated subcategory
of the derived category $\DA$.

The concept of silting complexes originated from \cite{kv}, and
has more recently been studied by many authors, often motivated by
combinatorial aspects related to mutations, as in \cite{ai}.
Moreover, the case of 2-term silting is of particular interest, see e.g.
\cite{air}, \cite{by} and \cite{ky}.

In the setting of 2-term silting, we have the following theorem:

\begin{theorem}\label{Main1}
Let $B = \End_{\DA}(\p)$, for a 2-term silting
complex $\p$ in $K^b(\proj A)$. Then the following hold.
\begin{itemize}
\item[(a)] If $\gld A =1$, then $\gld B \leq 3$.
\item[(b)] If $\gld A =2$, then $\gld B \leq 7$.
\end{itemize}
Moreover, for each $n > 2$, there is an algebra $A$, with $\gld A = n$, such that $K^b(\proj A)$ admits a 2-term
silting complex $\p$ with $\gld \End_{\DA}(\p) = \infty$.
\end{theorem}

Note that the projective presentation of a tilting $A$-module $T$ as defined above, gives rise to a 2-term silting complex $\p_T$ in $K^b(\proj A)$, and that we have an isomorphism of algebras $\End_A(T) \cong \End_{\DA}(\p_T)$.

The situation in part (a) was studied in \cite{bz2}. In this case
$B$ is a called a {\em silted algebra}, and it was proved that
silted algebras are so-called shod algebras \cite{cl}, in particular
this implies that $\gld B \leq 3$, by \cite{hrs}.

The main body of this paper is a proof of (b), an
example that the global dimension of $B$ actually can
be $7$ in this case, and a class of examples that justifies the last statement of Theorem \ref{Main1}.

We also prove that with a stronger assumption on $\p$,
we actually get that $\gld B$ is bounded by $\gld A$.
More precisely, we show the following.

\begin{theorem}\label{Main2}
With the above notation, and assuming in addition that $\pd H^0(\p) \leq 1$,
we have $\gld B \leq 2 (\gld A) +2$.
\end{theorem}

In the first section, we recall some notation and background concerning 2-term silting complexes and their endomorphism algebras.
In the second section, we prove some preliminary general results. Then, in Section 3 and 4, we prove respectively Theorem \ref{Main2}
and Theorem \ref{Main1}, while in the last section, we give some examples.

\section{Background and notation}
Let $A$ be a finite dimensional algebra with $\gld A=d$. Then $K^b(\proj A)=\DA:=\D$. Let $\p$ be a 2-term silting complex in $\D$ and let $B=\End_\D(\p)$. We recall some classical notation (see e.g. \cite{ass}) and some results from \cite{bz1}, which will be used freely in the remaining of the paper.

Recall that a pair of subcategories  $(\X,\Y)$ of $\mod A$,
is called a {\em torsion pair}, if the following hold:
\begin{itemize}
\item[-] $\Hom_A(\X, Y)= 0$ if and only if $Y$ is in $\Y$, and
\item[-] $\Hom_A(X, \Y)= 0$ if and only if $X$ is in $\X$.
\end{itemize}

For a given torsion pair $(\X,\Y)$ and an object $M$ in $\mod A$,
there is a (unique) exact sequence
$$0 \to tM \to M \to M/tM \to 0$$
with $tM$ in $\X$ and $M/tM$ in $\Y$. This is called the {\em canonical sequence} of $M$. Furthermore, for an $A$-module $X$
we let $\add X$ denote the additive closure of $X$ in
$\mod A$, and we
let $\Fac X$ denote the full subcategory of all quotients of
modules in $\add X$. The first notion is also used for a complex $X$ in $\D$.

For a 2-term silting complex $\p$,
consider the full subcategories of $\mod A$ given by
\begin{itemize}
\item[-]$\T(\p) =\{X \in \mod A \mid \Hom_{\D}(\p,X[1]) = 0$,  and
\item[-]$\F(\p) =\{Y \in \mod A \mid \Hom_{\D}(\p,Y) = 0$.
\end{itemize}
Furthermore, let $B = \End_{\D}(\p)$.
The following summarizes results from \cite{bz1} which will
be essential later in this paper.

\begin{proposition}\label{prop:summa}
Let $\p$ be a 2-term silting complex in $K^b(\proj A)$. Then the following hold.
\begin{itemize}
\item[(a)] The pair $(\T(\p),\F(\p))$ is a torsion pair in $\mod A$.
\item[(b)] $\T(\p) = \Fac H^0(\p)$.
\item[(c)] The category $\C(\p) = \{\x \in \D \mid  \Hom(\p,\x[i]) = 0 \text{  for  } i \neq 0 \}$ is an abelian category with
short exact sequences coinciding with the triangles in $\D$ whose vertices are in $\C(\p)$.
\item[(d)] Let $\x$ be in $\D$. Then we have that $\x$ is in $\C(\p)$ if and only if $H^0(\x)$ is in $\T(\p)$,
$H^{-1}(\x)$ is in $\F(\p)$ and $H^i(\x) = 0$ for $i \neq -1, 0$.
\item[(e)] $\Hom_{\D}(\p, - )\colon \C(\p) \to \mod B$ is an equivalence of (abelian) categories.
\end{itemize}
\end{proposition}

For full subcategories $\X$ and $\Y$ of $\D$, we let $\X \ast \Y$ denote the full subcategory of $\D$ with objects
$Z$ appearing in a triangle
$$X \to Z \to Y \to X[1]$$ with $X$ in $\X$ and $Y$ in $\Y$.
It follows from the octahedral axiom that we have $(\X \ast \Y) \ast \Z = \X \ast (\Y \ast \Z)$, for
three full subcategories $\X, \Y$ and $\Z$. The subcategory $\X$ is called {\em extension closed} if
$\X \ast \X = \X$.
We will need the following fact, which follows from \cite[Propositions 2.1 and 2.4]{iy}.

\begin{lemma}\label{lem:iy}
Let $\X_i$ be subcategories of $\D$, with $\Hom_\D(\X_i, \X_j) = 0 =Hom_\D(\X_i, \X_j[1])$ for $i<j$.
Then $\X_1 \ast \X_2 \ast \cdots \ast \X_n$ is closed under extensions and direct summands.
\end{lemma}

\section{Preliminaries}\label{sec:pre}

Now, fix a 2-term silting complex $\p$ in $K^b(\proj A)$, and let $\P = \add \p$.
In this section we include some general observations on projective
objects and projective dimensions in $\CP$.

For each $\p_0$ in $\P$, given by $P_0^{-1} \xrightarrow{p_0} P_0^0$, consider the canonical exact sequence of $H^{-1}(\p_0)$ relative to the torsion pair $(\T(\p),\F(\p))$:
\[0\rightarrow tH^{-1}(\p_0)\rightarrow H^{-1}(\p_0)\rightarrow H^{-1}(\p_0)/tH^{-1}(\p_0)\rightarrow 0.\]
So $tH^{-1}(\p_0)$ is a submodule of $P_0^{-1}$ and we denote by $\pi \colon P_0^{-1}\rightarrow P_0^{-1}/tH^{-1}(\p_0)$ the canonical epimorphism.
Let $\widetilde{\p}_0$ be the complex $P_0^{-1}/tH^{-1}(\p_0)\s{\widetilde{p}_0}\rightarrow P_0^0$, where $\widetilde{p}_0$ is the unique homomorphism such that the diagram
\[
\xymatrix{
                & P_0^{-1}/tH^{-1}(\p_0) \ar[dr]^{\widetilde{p}_0}             \\
 P_0^{-1} \ar[ur]^{\pi} \ar[rr]^{p_0} & &     P_0^0        }
\]
commutes.

Let $\P_C =\P \cap \CP$.

\begin{lemma}\label{lem:pc}
Let $\p_0$ be in $\P$. Then $\p_0$ is in $\P_C$
if and only if $\p_0 \cong \widetilde{\p}_0$.
\end{lemma}

\begin{proof}
We have by definition that $\p_0 \cong \widetilde{\p}_0$
if and only if $tH^{-1}(\p_0) = 0$ if and only if
$H^{-1}(\p_0)$ is in $\F(\p)$ if and only if
$\Hom(\p, H^{-1}(\p_0)) = 0$. It is straightforward to
check that $\Hom(\p, H^{-1}(\p_0)) = 0$ if and only if
$\Hom(\p, \p_0[-1]) = 0$. Moreover, we have
that $\Hom(\p, \p_0[-1]) = 0$ if and only if $\p_0$ is in $\CP$,
and the statement follows from this.
\end{proof}

\begin{lemma}\label{lem:proj1}
With notation as above, the following hold.
\begin{itemize}
\item[(a)] There is a triangle in $\D$:
$$tH^{-1}(\p)[1]\rightarrow \p \rightarrow \widetilde{\p} \rightarrow tH^{-1}(\p)[2].$$

\item[(b)] The object $\widetilde{\p}$ is a projective generator
for $\CP$.
\end{itemize}
\end{lemma}

\begin{proof}
The triangle in (a) exists by the construction of $\widetilde{\p}$.

Note that $H^0(\widetilde{\p})=H^0(\p)$ is in $\T(\p)$ and $H^{-1}(\widetilde{\p})=H^{-1}(\p)/tH^{-1}(\p)$ is in $\F(\p)$.
Then by Proposition~\ref{prop:summa} (d), we have $\widetilde{\p}\in\CP$.
Applying the functor $\Hom_\D(\p,-)$ to this triangle yields an isomorphism
\[\Hom_\D(\p,\p)\cong\Hom_\D(\p,\widetilde{\p})\]
as $B$-modules. Now (b) follows from Proposition \ref{prop:summa} (e).

%By Theorem~\ref{prop:known} (1), the objects $\widetilde{P}_i$, %$1\leq i\leq n$, are the non-isomorphic projective objects in $\CP$.
\end{proof}

For any integer $i$, we let $\mathcal{\D}^{\leq i}(\p)
= \{\x \in \D \mid \Hom_{\D}(\p,\x[j]) =0 \text{  for  } j>i   \}$,
and we let $\mathcal{\D}^{\geq i}(\p) = \{\x \in \D \mid \Hom_{\D}(\p,\x[j]) =0 \text{  for  } j<i   \}$.

\begin{lemma}\label{lem:CPP}
With notation as above, we have:
$\CP\subset\P\ast\P[1]\ast\cdots\ast\P[d+1]$.
\end{lemma}

\begin{proof}
By \cite[Proposition 2.23]{ai}, we have $$\CP\subset\D^{\leq0}(\p)\subset\P\ast\P[1]\ast\cdots\ast\P[l-1]\ast\P[l]$$ for some $l>0$.
For any $\m$ in $\CP$, by Proposition~\ref{prop:summa} (d), we have $H^i(\m)=0$ for $i\neq-1,0$. So there is a complex $\x$ of projective $A$-modules, which is equivalent to $\m$, and such that $H^i(\x)=0$ for $i>0$ or $i<-d-1$. So
$$\Hom_\D(\m,\p[i])\cong\Hom_\D(\x,\p[i])=0, \text{ }i\geq d+2,$$
which implies that $\m$ is in $\P\ast\P[1]\ast\cdots\ast\P[d+1].$
\end{proof}

\begin{lemma}\label{lem:pdn+1}
For a complex $\x$ in $\CP \cap (\P_C\ast\P_C[1]\ast\cdots\P_C[m])$ for some $m\geq0$, we have $\pd\Hom_\D(\p,\x)_B\leq m$.
\end{lemma}

\begin{proof}
Let $\x_0 =\x$. There are triangles
$$\x_{i+1} \rightarrow \oo_i \xrightarrow{g_i} \x_i \rightarrow \x_{i+1}[1], \ \ \ 0\leq i\leq m-1$$
where $\oo_i$ is in $\P_C$ and
%where $g_i$ is a minimal right $\P_C$-approximation of
$\x_i$ is in $\P_C\ast\P_C[1]\ast\cdots\ast\P_C[m-i]$.
Since $\Hom_\D(\p,\p[i])=0$ for all $i>0$, we have that $g_i$ is a right $\P$-approximation of $\x_i$. By Lemma \ref{lem:pc} and Lemma \ref{lem:proj1} (b), each $\oo_i$ is projective in $\CP$.
Assume that $\x_i$ is in $\CP$ for some $0\leq i\leq m-1$. Then,
since $g_i$ is a right $\P-$approximation and $\oo_i$ is projective in
$\CP$, we have that $g_i$ is an epimorphism in $\CP$.
So $\x_{i+1}$ is the kernel of $g_i$, by Proposition
\ref{prop:summa} (c).
Note that $\x_0\in\CP$. Then by induction on $i$, we have that $\x_i\in\CP$ for all $0\leq i\leq m$ and
\[\pd\Hom_\D(\p,\x_i)_B\leq\pd\Hom_\D(\p,\x_{i+1})_B+1.\]
Therefore $\pd\Hom_\D(\p,\x)_B\leq\pd\Hom_\D(\p,\x_{m})_B +m=m$ since $\x_{m}\in\P_C$ is projective in $\CP$.
\end{proof}

We end this section by considering the following special case. Recall from \cite{rick}, that a 2-term silting complex $\p$ in $K^b(\proj A)$ is a {\em tilting complex} if
$\Hom_{\D}(\p,\p[-1]) =0$.

\begin{proposition}\label{prop:tilt}
If the 2-term silting complex $\p$ is a tilting complex, then $\gld\End_\D(\p)\leq \gld A+1$.
\end{proposition}

\begin{proof}
If $\p$ is tilting, then $\p$ is in $\CP$. So we infer that $\P=\P_C$. It follows from Lemma~\ref{lem:CPP} and Lemma~\ref{lem:pdn+1} that $\gld\End_\D(\p)\leq \gld A+1$.
\end{proof}

Note that the classical situation (as in \cite[III, section 3.4]{ha})
where $\p$ is the projective resolution of a classical tilting
module,
is covered by this result.

\section{The partial tilting case}

Throughout this section, we assume that $\pd H^0(\p)_A \leq 1$, that is: $H^0(\p)$ is a partial tilting $A$-module. Then we have that $Q=H^{-1}(\p)$ is projective as an $A$-module,
and $\p\cong H^0(\p)\oplus Q[1]$. Consider the canonical exact sequence of $Q$ relative to the torsion pair $(\T(\p),\F(\p))$:
\[0\rightarrow tQ\rightarrow Q\rightarrow Q/tQ\rightarrow 0.\]
As before, we let $d = \gld A$.
We first prove two technical lemmas.

\begin{lemma}\label{lem:XQ1}
With the above notation, we have
$$tQ\in\add H^0(\p)\ast\add H^0(\p)[1]\ast\cdots\ast\add H^0(\p)[d-1].$$

\end{lemma}
\begin{proof}
We first note that $tQ\in\T(\p)$, so by definition $\Hom_\DA(\p,tQ[i])=0$ for $i\neq 0$. In particular, we have $\Hom_\DA(Q[1],tQ[i])=0$ for $i\neq 0$. For $i=0$, since both $Q$ and $tQ$ are $A$-modules, we also have that $\Hom_\DA(Q[1],tQ)=0$. It follows from $T(\p)\in\C(\p)$ that by Proposition \ref{prop:summa}, we have $tQ\in\P\ast\P[1]\ast\cdots\ast\P[d+1]$. Therefore, using $\p\cong H^0(\p)\oplus Q[1]$, we get that $tQ$ is in
$\add H^0(\p)\ast\add H^0(\p)[1]\ast\cdots\ast\add H^0(\p)[d+1]$.
By the canonical sequence of $Q$, we have $\pd (tQ)_A \leq\pd (Q/tQ)_A -1 \leq d-1$. Hence, it follows that $\Hom(tQ,\p[d])= 0= \Hom(tQ,\p[d+1])$. The claim of the lemma
follows.
\end{proof}

\begin{lemma}\label{lem:CPP1}
With the above notation, we have
$\C(\p)\subset\P\ast\P[1]\ast\cdots\ast\P[d]\ast\add H^0(\p)[d+1].$
\end{lemma}
\begin{proof}
Using that $\p\cong H^0(\p)\oplus Q[1]$ in combination with Lemma
\ref{lem:CPP}, we only need to prove that $\Hom_\D(\x,Q[d+2])=0$ for $\x\in\CP$. This follows from $\pd H^i(\x)_A\leq d$ for $i=-1,0$ and $H^i(\x)=0$ for $i\neq-1,0$.
\end{proof}

We can now prove the main result of this section.

\begin{theorem}\label{thm:pd1}
If $\pd (H^0(\p))_A \leq 1$, then $\gld B\leq 2\gld A+2$.
\end{theorem}

\begin{proof}
Let $\x$ be an object in $\CP$ with
$$\x\in\P\ast\cdots\ast\P[i]\ast\add H^0(\p)[i+1]\ast\cdots\ast\add H^0(\p)[d+1]$$
for some $0\leq i\leq d$. Then there is a triangle
$$\x_1 \rightarrow \e \s{g_{\x}}{\rightarrow} \x \rightarrow \x_1[1]$$
where $g_{\x}$ is a right $\P-$approximation of $\x$ and
$\x_1$ is in $$\P\ast\cdots\ast\P[i-1]\ast\add H^0(\p)[i]\ast\cdots\ast\add H^0(\p)[d].$$
Then $\Hom_\D(\p,g_{\x})$ is an epimorphism
and $\Hom_\D(\p,\e)$ is projective in $\mod B$.

Recall that $Q = H^{-1}(\p)$. Then,
by Lemma~\ref{lem:proj1} there is a triangle
\[F[1]\rightarrow \e\rightarrow \widetilde{\e} \rightarrow F[2]\]
where $F$ is in $\add tQ\subset\T(\p)\subset\CP$ and $\widetilde{\e}$ is projective in $\CP$. So $\Hom_\D(F[1],\x)=0$ since $\x\in\C(\p)$. It follows that the map $g_{\x}$ factors through the map $\e \rightarrow \widetilde{\e}$. Then,
by the octahedral axiom, we have the following commutative diagram of triangles:
$$\xymatrix{
&\x[-1]\ar@{=}[r]\ar[d]&\x[-1]\ar[d]\\
F[1]\ar[r]\ar@{=}[d]&\x_1\ar[r]\ar[d]&\x'\ar[r]\ar[d]&F[2]\ar@{=}[d]\\
F[1]\ar[r]&\e\ar[r]\ar[d]^{g_{\x}}&
\widetilde{\e}\ar[r]\ar[d]^{\widetilde{g}_{\x}}&F[2]\\
&\x\ar@{=}[r]&\x
}.$$
Then we have that
$$
\begin{array}{rcl}
\x'&\in&\add \x_1\ast\add F[2]\\
&\subset&\left(\P\ast\cdots\ast\P[i-1]\ast\add H^0(\p)[i]\ast\cdots\ast\add H^0(\p)[d]\right)\\
&&\ast\left(\add H^0(\p)\ast\cdots\ast\add H^0(\p)[d-1]\right)[2]\\
&=&\left(\P\ast\cdots\ast\P[i-1]\ast\add H^0(\p)[i]\ast\cdots\ast\add H^0(\p)[d]\right)\ast\add H^0(\p)[d+1]
\end{array}$$
where the inclusion is due to Lemma~\ref{lem:XQ1},
and the equality follows from $$\P\ast\cdots\ast\P[i-1]\ast\add H^0(\p)[i]\ast\cdots\ast\add H^0(\p)[d]$$ being
closed under extensions by Lemma \ref{lem:iy}.
Applying $\Hom_\D(\p,-)$ to the above diagram, we obtain a commutative diagram
$$
\xymatrix@C=0.3pc{
\Hom_\D(\p,\e)\ar[rr]^{\cong}\ar[dr]_{\Hom_\D(\p,g_{\x})}&&\Hom_\D(\p,\widetilde{\e})
\ar[dl]^{\Hom_\D(\p,\widetilde{g}_{\x})}\\
&\Hom_\D(\p,\x)&
}
$$
%then $\Hom_\D(\p,\widetilde{g}_{\x})$ is also a projective cover of %$\Hom_\D(\p,\x)$.
Using that the map $\Hom_\D(\p,g_{\x})$ is an epimorphism in
$\mod B$, it follows that the map $\widetilde{g}_{\x}$ is an epimorphism
in $\C(\p)$.
Then $\x'$ is the kernel of $\widetilde{g}_{\x}$ in $\CP$. Hence
$$\pd\Hom_\D(\p,\x)_B\leq\pd \Hom_\D(\p,\x')_B +1.$$
Using induction on $i$ and Lemma~\ref{lem:CPP1}, we have that for $\x\in\CP$, there is $\x'$ such that
\begin{equation}\label{l1}\x'\in\CP\cap\left(\add H^0(\p)\ast\add H^0(\p)[1]\ast\cdots\ast\add H^0(\p)[d+1]\right)\end{equation}
and $\pd\Hom_\D(\p,\x)_B \leq\pd \Hom_\D(\p,\x')_B+d+1$.
By Lemma~\ref{lem:pdn+1} and equation (\ref{l1}), we have $\pd\Hom_\D(\p,\x')_B \leq d+1$. It then follows that
$\pd\Hom_\D(\p,\x)_B\leq 2d+2$, and hence $\gld B\leq 2 d+2$.
\end{proof}

\section{The case of global dimension 2.}

In this section, we consider the case when $\gld A\leq 2$.
Our aim is to prove part (b) of Theorem \ref{Main1}, stating
that in this case we have that the global dimension is at most 7
for the endomorphism algebra of any 2-term silting complex.

We prepare by showing four technical lemmas.
Let $\P_C^{[0,1]} =\left(\P\ast\P[1]\right)\cap\C(\p)$.

\begin{lemma}\label{lem:pd1}
If $\x$ is in $\P_\C^{[0,1]}$, then $\pd \Hom_\D(\p,\x)_B \leq 1.$
\end{lemma}

\begin{proof}
Since $\x$ is in $\P\ast\P[1]$, there is a triangle $\oo_1\rightarrow \oo_0\rightarrow \x\rightarrow \oo_1[1]$ with $\oo_0,\oo_1\in\P$. Applying the functor $\Hom_\D(\p,-)$ to this triangle, we get a long exact sequence
$$
\Hom_\D(\p,\x[-1]) \rightarrow \Hom_\D(\p,\oo_1) \rightarrow \Hom_\D(\p,\oo_0)
\rightarrow \Hom_\D(\p,\x) \rightarrow \Hom_\D(\p,\oo_1[1])
$$
where the first term is zero since $\x$ is in $\C(\p)$, and the last term is zero since $\Hom_\D(\p,\p[1])=0$. Therefore, $\pd \Hom_\D(\p,\x)_B\leq 1.$
\end{proof}

\begin{lemma}\label{lem:pd3}
If $\x$ is in $\CP \cap (\P_\C^{[0,1]}\ast\P_\C^{[0,1]}[1]\ast\P_\C^{[0,1]}[2])$, then
$\pd \Hom_\D(\p,\x)_B\leq3.$
\end{lemma}

\begin{proof}
By $\x \in\P_\C^{[0,1]}\ast\P_\C^{[0,1]}[1]\ast\P_\C^{[0,1]}[2]$, there are triangles
\begin{equation}\label{trian1} \Ll \rightarrow \d_1\rightarrow \x \rightarrow \Ll[1]\end{equation}
and
\begin{equation}\label{trian2}\d_3\rightarrow \d_2\rightarrow \Ll \rightarrow \d_3[1]\end{equation}
with $\d_1, \d_2, \d_3\in\P_\C^{[0,1]}$ and $\Ll \in\P_\C^{[0,1]}\ast\P_\C^{[0,1]}[1]\subset\P\ast\P[1]\ast\P[2]$. Applying $\Hom_\D(\p,-)$ to triangle (\ref{trian1}), we obtain a long exact sequence
\begin{multline*}
\Hom_\D(\p,\x[-2]) \rightarrow \Hom_\D(\p,\Ll[-1]) \rightarrow\Hom_\D(\p,\d_1[-1])\rightarrow\Hom_\D(\p,\x[-1])\\
\rightarrow \Hom_\D(\p,\Ll) \rightarrow\Hom_\D(\p,\d_1)\rightarrow\Hom_\D(\p,\x)
\rightarrow\Hom_\D(\p,\Ll[1]).
\end{multline*}
We have $\Hom_\D(\p,\x[i])=0$ for $i=-1$ or $i=-2$, since $\x$ is in
$\CP$. Furthermore, we have $\Hom_\D(\p,\d_1[-1])=0$, by $\d_1\in\CP$ and $\Hom_\D(\p,\Ll[1])=0$ by $\Ll \in\P\ast\P[1]\ast\P[2]$. From this it follows that we have a short exact sequence
$$0\rightarrow\Hom_\D(\p,\Ll) \rightarrow \Hom_\D(\p,\d_1) \rightarrow \Hom_\D(\p,\x)\rightarrow0$$
and that $\Hom_\D(\p,\Ll[-1])=0$. Using this short exact sequence, it follows that
\begin{equation}\label{ineq1} \pd \Hom_\D(\p,\x)_B\leq\max\{\pd \Hom_\D(\p,\d_1)_B, \pd \Hom_\D(\p,\Ll)_B +1\}.\end{equation}
Applying $\Hom_\D(\p,-)$ to the triangle (\ref{trian2}), we obtain an exact sequence
$$0=\Hom_\D(\p,\Ll[-1])\rightarrow\Hom_\D(\p,\d_3)\rightarrow\Hom_\D(\p,\d_2)\rightarrow
\Hom_\D(\p,\Ll)\rightarrow\Hom_\D(\p,\d_3[1])$$
where the last term is zero due to $\d_3\in\P_\C^{[0,1]}$. As above,
we obtain that
\begin{equation}\label{ineq2} \pd \Hom_\D(\p,\Ll)_B \leq\max\{\pd \Hom_\D(\p,\d_2)_B, \pd \Hom_\D(\p,\d_3)_B+1\}.\end{equation}
Now, combining the inequalities (\ref{ineq1}) and (\ref{ineq2}) with Lemma~\ref{lem:pd1}, we obtain $\pd \Hom_\D(\p,\x)_B \leq3.$
\end{proof}

\begin{lemma}\label{lemma}
If $\n$ is in $\D^{\geq-1}(\p) \cap \left(\P\ast\P[1]\ast\P[2]\right)$, then there is an object $\widetilde{\n}\in\C(\p)$ such that $\Hom_\D(\p,\n)\cong\Hom_\D(\p,\widetilde{\n})$ as $B-$modules and $\widetilde{\n} \in\add \n\ast\P_C^{[0,1]}[2]$.
\end{lemma}

\begin{proof}
Since $(\D^{\leq0}(\p),\D^{\geq0}(\p))$ is a $t$-structure (see \cite[Lemma 5.10]{ky}), there is a triangle
\begin{equation}\label{t-tria}
\m\rightarrow \n \rightarrow \widetilde{\n}\rightarrow \m[1]\end{equation}
with $\m \in\D^{\leq0}(\p)[1]$ and $\widetilde{\n}\in\D^{\geq0}(\p)$. Then $\m\in\P[1]\ast\cdots\ast\P[l]$ for some $l$ by \cite[Proposition 2.23]{ai}.
Applying the functor $\Hom_\D(\p,-)$ to the triangle (\ref{t-tria}), we have a long exact sequence
$$
\cdots\rightarrow\Hom_\D(\p,\m[i])\rightarrow\Hom_\D(\p,\n[i])\rightarrow
\Hom_\D(\p,\widetilde{\n}[i])\rightarrow\Hom_\D(\p,\m[i+1])\rightarrow\cdots
$$
Since $\Hom_\D(\p,\m[i])=0$ for $i\geq0$ by $\m\in\P[1]\ast\cdots\ast\P[l]$ and $\Hom_\D(\p,\widetilde{\n}[i])=0$ for $i<0$ by $\widetilde{\n}\in\D^{\geq0}(\p)$, and also  $\Hom_\D(\p,\n[i])=0$ for $i\neq-1,0$ by the assumption $\n\in\D^{\geq-1}(\p) \cap \left(\P\ast\P[1]\ast\P[2]\right)$, we have that \[\Hom_\D(\p,\n)\cong\Hom_\D(\p,\widetilde{\n})\] as $B-$modules,\[\Hom_\D(\p,\widetilde{\n}[i])=0, \text{ for }i>0,\] and \[\Hom_\D(\p,\m[i])=0\text{, for }i<-1.\]
Thus, we obtain $\widetilde{\n}\in\D^{\leq0}(\p)\cap\D^{\geq0}(\p)=\CP$ and $\m\in\D^{\geq0}(\p)[1]\cap\D^{\leq0}(\p)[1]=\CP[1]$. Then by Lemma~\ref{lem:CPP}, we have $\widetilde{\n}\in\P\ast\P[1]\ast\P[2]\ast\P[3]$. Applying the functor $\Hom_\D(-,\p)$ to the triangle (\ref{t-tria}), we obtain a long exact sequence
\[
\cdots\rightarrow\Hom_\D(\widetilde{\n},\p[i])\rightarrow\Hom_\D(\n,\p[i])\rightarrow
\Hom_\D(\m,\p[i])\rightarrow\Hom_\D(\widetilde{\n},\p[i+1])\rightarrow\cdots
\]
We have $\Hom_\D(\n,\p[i])=0$ for $i>2$ by $\n \in \P \ast \P[1] \ast \P[2]$, and we have $\Hom_\D(\widetilde{\n},\p[i])=0$ for $i>3$ by $\widetilde{\n}\in\P\ast\P[1]\ast\P[2]\ast\P[3]$. From this it follows
that $\Hom_\D(\m,\p[i])=0$ for $i>2$, and hence
$\m\in(\P[1]\ast\P[2])\cap\CP[1]=\P_C^{[0,1]}[1]$. Therefore we have that
\[\widetilde{\n}\in\add \n \ast\add \m[1]\subset\add \n\ast\P_C^{[0,1]}[2].\]
\end{proof}

\begin{lemma}\label{lem:ind}
Let $\x\in\C(\p)\cap\left(\P\ast\P[1]\ast\cdots\ast\P[t]\ast\H[t+1]\right)$ for some $t$ with $0\leq t\leq3$, where $\H\subset\P\ast\P[1]\ast\cdots\ast\P[2-t]$ (and $\H=0$ for $t=3$). Then for each $r$ with $0\leq r\leq \min\{t+1,3\}$, there is an object $\widetilde{\x}_r\in\CP$ such that
\[\pd\Hom_\D(\p,\x)_B\leq\pd\Hom_\D(\p,\widetilde{\x}_r)_B+r\]
and
\[\widetilde{\x}_r\in\P\ast\P[1]\ast\cdots\ast\P[t-r]\ast\H[t+1-r]\ast\P_\C^{[0,1]}[3-r]\ast\cdots\ast\P_\C^{[0,1]}[2]\]
where $\P\ast\P[1]\ast\cdots\ast\P[t-r]$ is taken to be 0 when $r=t+1$ and $\P_\C^{[0,1]}[3-r]\ast\cdots\ast\P_\C^{[0,1]}[2]$ is taken to be 0 when $r=0$.
\end{lemma}

\begin{proof}
Let $\widetilde{\x}_0=\x$. Then $\widetilde{\x}_0$ satisfies the conditions in the lemma. Assume that $\widetilde{\x}_{r-1}$ satisfying the conditions. By
$$\widetilde{\x}_{r-1}\in\P\ast\P[1]\ast\cdots\ast\P[t-(r-1)]\ast\H[t+1-(r-1)]\ast\P_\C^{[0,1]}[3-(r-1  )]\ast\cdots\ast\P_\C^{[0,1]}[2],$$ there is a triangle
\[\x_r\rightarrow \p_0\rightarrow \widetilde{\x}_{r-1}\rightarrow \x_r[1]\]
with $\p_0\in\P$, $\x_r\in\P\ast\cdots\ast\P[t-r]\ast\H[t+1-r]\ast\P_\C^{[0,1]}[3-r]\ast\cdots\ast\P_\C^{[0,1]}[1]\subset\P\ast\P[1]\ast\P[2]$. The inclusion follows from Lemma \ref{lem:iy}.
Applying $\Hom_\D(\p,-)$ to this triangle, we have a long exact sequence
\[\cdots\rightarrow\Hom_\D(\p,\x_r[i])\rightarrow\Hom_\D(\p,\p_0[i])\rightarrow
\Hom_\D(\p,\widetilde{\x}_{r-1}[i])\rightarrow\Hom_\D(\p,\x_r[i+1])\rightarrow\cdots   \]
Since $\Hom_\D(\p,\widetilde{\x}_{r-1}[i])=0$ for $i\neq0$ by $\widetilde{\x}_{r-1}\in\CP$, $\Hom_\D(\p,\x_r[1])=0$
by $\x_r \in \P\ast\P[1]\ast \P[2]$, and also $\Hom_\D(\p,\p_0[i])=0$ for $i<-1$ by $\p$ being 2-term, we have a short exact sequence
\[0\rightarrow\Hom_\D(\p,\x_r)\rightarrow\Hom_\D(\p,\p_0)\rightarrow\Hom_\D(\p,\widetilde{\x}_{r-1})\rightarrow0\]
and
\[\Hom_\D(\p,\x_r[i])=0\text{ for $i<-1$.}\]
Then $\pd\Hom_\D(\p,\widetilde{\x}_{r-1})_B\leq\pd \Hom_\D(\p,\x_r)_B+1$
and by Lemma \ref{lemma}, there is an object $\widetilde{\x}_r\in\C(\p)$ such that $\Hom_\D(\p,\widetilde{\x}_r)_B \cong\Hom_\D(\p,\x_r)_B$ and $$\widetilde{\x}_r\in\add \x_r\ast\P^{[0,1]}_\C[2]\subset\P\ast\P[1]\ast\cdots\ast\P[t-r]\ast\H[t+1-r]\ast\P_\C^{[0,1]}[3-r]\ast\cdots\ast\P_\C^{[0,1]}[1]\ast\P^{[0,1]}_\C[2].$$
\end{proof}

Now we prove the main result in this section.

\begin{theorem}\label{thm:2-silt}
If $\gld A\leq 2$, then $\gld \End_\D(\p)\leq 7$ for any 2-term silting complex $\p$ in $K^b(\proj A)$.
\end{theorem}

\begin{proof}
Let $\x\in\CP$. Then by Lemma~\ref{lem:CPP}, we have that $\x\in\P\ast\P[1]\ast\P[2]\ast\P[3]$. By Lemma~\ref{lem:ind}, (taking $t=3$, $r=2$ and hence $\H =0$), there is an $\widetilde{\x}\in\CP$ such that $\widetilde{\x} \in\P\ast\P[1]\ast\P^{[0,1]}_\C[1]\ast\P^{[0,1]}_\C[2]$, and
\begin{equation}\label{n1}
\pd \Hom_\D(\p,\x)_B\leq\pd\Hom_\D(\p,\widetilde{\x})_B+2.
\end{equation}
Then there is a triangle
\[\z\rightarrow \y\rightarrow \widetilde{\x}\rightarrow \z[1]\]
with $\y\in\P\ast\P[1]$ and $\z\in\P^{[0,1]}_\C\ast\P^{[0,1]}_\C[1]$.
Applying the functor $\Hom_\D(\p,-)$ to this triangle, we have a long exact sequence
\[\cdots\rightarrow\Hom_\D(\p,\z[i])\rightarrow\Hom_\D(\p,\y[i])\rightarrow
\Hom_\D(\p,\widetilde{\x}[i])\rightarrow\Hom_\D(\p,\z[i+1])\rightarrow\cdots\]
Since $\Hom_\D(\p,\widetilde{\x}[i])=0$ for $i\neq0$ by $\widetilde{\x}\in\CP$, and $\Hom_\D(\p,\z[i])=0$ for $i\neq-1,0$ by $\z\in\C(\p)\ast\C(\p)[1]$, we have a short exact sequence
\[0\rightarrow\Hom_\D(\p,\z)\rightarrow\Hom_\D(\p,\y)\rightarrow\Hom_\D(\p,\widetilde{\x})\rightarrow0,\]
and
\[\Hom_\D(\p,\y[i])=0\text{ for $i<-1$.}\]
Then we have that
\begin{equation}\label{n2}
\pd\Hom_\D(\p,\widetilde{\x})_B\leq\max\{\pd\Hom_\D(\p,\y)_B,\pd\Hom_\D(\p,\z)_B+1\}
\end{equation}
and $\y, \z \in \mathcal{D}^{\geq -1}(\p)$.
By Lemma \ref{lemma}, there are objects $\widetilde{\y}$, $\widetilde{\z}\in\C(\p)$ such that:
$$\begin{tabular}{lll}
$\widetilde{\y}\in\P\ast\P[1]\ast\P^{[0,1]}_\C[2]$ & &
$\widetilde{\z}\in\P^{[0,1]}_\C\ast\P^{[0,1]}_\C[1]\ast\P^{[0,1]}_\C[2]$ \\
& & \\
$\Hom_\D(\p,\widetilde{\y}_B)\cong\Hom_\D(\p,\y)_B$   & &
$\Hom_\D(\p,\widetilde{\z})_B\cong\Hom_\D(\p,\z)_B$
\end{tabular}
$$
By Lemma~\ref{lem:ind}, (taking $t=1$, $r=2$ and $\H=\P^{[0,1]}_\C[2]$), there is an object $\widetilde{\y'}\in\CP$ such that $\widetilde{\y'}\in\P^{[0,1]}_\C\ast\P^{[0,1]}_\C[1]\ast\P^{[0,1]}_\C[2]$ and
\begin{equation}\label{n3}
\pd\Hom_\D(\p,\widetilde{\y})_B\leq\pd\Hom_\D(\p,\widetilde{\y'})_B
+2.
\end{equation}
By Lemma~\ref{lem:pd3}, we have $\pd\Hom_\D(\p,\widetilde{\z})_B\leq 3$ and $\pd\Hom_\D(\p,\widetilde{\y'})_B\leq 3$. Hence, combining
(\ref{n1}), (\ref{n2}) and (\ref{n3}), we obtain
\[\begin{array}{rcl}
\pd\Hom_\D(\p,\x)_B&\leq&\pd\Hom_\D(\p,\widetilde{\x})_B+2\\
&\leq&\max\{\pd\Hom_\D(\p,\y)_B,\pd\Hom_\D(\p,\z)_B +1\}+2\\
&=&\max\{\pd\Hom_\D(\p,\widetilde{\y})_B,
\pd\Hom_\D(\p,\widetilde{\z})_B+1\}+2\\
&\leq&\max\{\pd\Hom_\D(\p,\widetilde{\y'})_B+2,\pd\Hom_\D(\p,\widetilde{\z})_B+1\}+2\\
&\leq&7.
\end{array}\]
\end{proof}

\section{Examples}

\subsection{First example}

We first give an example to show that the bound in Theorem~\ref{thm:pd1} is possible.
Let $n\geq 2$ and $A=\k Q/I$ where $Q$ is the following quiver
\[\xymatrix@R=1pc{
2n+3\ar[r]^{a_{n}}&\cdots\ar[r]^{a_3}&7\ar[r]^{a_2}&5\ar[r]^{a_1}&3\ar[drr]^{c_1}\ar[dr]_{c_2}\\
&&&&&2\ar[dl]_{d_2}&1\ar[dll]^{d_1}\\
2n+4&\cdots\ar[l]_{b_{n}}&8\ar[l]_{b_3}&6\ar[l]_{b_2}&4\ar[l]_{b_1}
}\]
and the ideal $I$ is generated by $c_1d_1-c_2d_2$, $a_{i+1}a_i$ and $b_ib_{i+1}$, $1\leq i\leq n-1$. Then $\gld A=n$. Let $\p$ be the direct sum of the following complexes in $K^b(\proj A)$:
\[\begin{array}{cccl}
0&\longrightarrow&\bigoplus_{1\leq i\leq n+2,\ i\neq 2}P_{2i}&,\\
P_4&\longrightarrow&P_2&,\\
P_1&\longrightarrow&P_3&,\\
\bigoplus_{1\leq i\leq n+2,\ i\neq 2}P_{2i-1}&\longrightarrow&0&.
\end{array}\]
It is easily verified that $\p$ is a 2-term silting complex.
The quiver of the endomorphism ring $\End_\D(\p)$ is the Dynkin quiver of type $A_{2n+4}$:
\[\xymatrix@R=1pc{
2n+3\ar[r]^{a_{n}}&\cdots\ar[r]^{a_3}&7\ar[r]^{a_2}&5\ar[r]^{a_1}&3\ar[r]^{a_0}&1\ar[d]^c\\
2n+4&\cdots\ar[l]_{b_{n}}&8\ar[l]_{b_3}&6\ar[l]_{b_2}&4\ar[l]_{b_1}&2\ar[l]_{b_0}
}\]
with the relations $a_{i+1}a_i=0$, $b_ib_{i+1}=0$, $1\leq i\leq n-1$ and $a_0cb_0=0$. Hence the global dimension of $\End_\D(\p)$ is $2n+2$.

\subsection{Second example}

The next example shows that $7$ is a possible value for the global dimension of
the endomorphism algebra of a 2-term silting complex over an algebra of
global dimension two. Let $A=\k Q/I$ with $Q$ the following quiver
\[\xymatrix{
1\ar[r]^{a_1}&2\ar[r]^{a_2}&3\ar[r]^{a_3}&4\ar@<-.5ex>[r]_{a_5} \ar@<.5ex>[r]^{a_4}&5\ar[r]^{a_6}&6\ar[r]^{a_7}&7\ar[r]^{a_8}&8
}\]
and $I$ the ideal generated by $a_1a_2$, $a_3a_4a_6$ and $a_7a_8$.
Then $A$ has global dimension two, and the
 complex $\p$ given by the direct sum of the complexes
\[\begin{array}{cccl}
0&\longrightarrow&\bigoplus_{i=5,7,8}P_{i}&,\\
P_6&\longrightarrow&P_5&,\\
P_4&\longrightarrow&P_3&,\\
\bigoplus_{i=1,2,4}P_{i}&\longrightarrow&0&.
\end{array}\]
is a 2-term silting complex.
It is easily verified, that the quiver of $\End_\D(\p)$ is a linearly oriented Dynkin quiver of type $A_8$ and the ideal of relation equals the square of the Jacobson radical. Hence the global dimension of $\End_\D(\p)$ is 7.

\bigskip

\subsection{Third example}

The last example shows that there is no bound on the global dimension of the endomorphism algebra of a 2-silting object over an
algebra with global dimension $d\geq 3$. This example then completes the proof of Theorem \ref{Main1}.

Let first $A=\k Q/I$ where the quiver $Q$ is
\[
\xymatrix@C=2pc@R=2pc{
3\ar@<-.5ex>[r]_b&2\ar@<-.5ex>[l]_a\ar[d]^c\ar@<-.5ex>[r]_d&4\ar@<-.5ex>[l]_e\\
&1&
}\]
and $I=\langle ba, bd, abc, de\rangle$. The indecomposable projective $A-$modules are
\[P_1=\begin{smallmatrix}1\end{smallmatrix},\ \ \  P_2=\begin{smallmatrix}\ \ 2\ \ \\1\ 3\ 4\\\ \ 2\ \ \end{smallmatrix},
\ \ \  P_3=\begin{smallmatrix}3\\2\\1\end{smallmatrix},\ \ \  P_4=\begin{smallmatrix}\ \ 4\ \ \\\ \ 2\ \ \\1\ 3\ 4\\\ \ 2\ \ \end{smallmatrix}.\]
The integers here denote the corresponding simples, and the notation indicates the radical filtration. The global dimension of $A$ is 3. Let $\p$ be the direct sum of
\[
\p_i=\cdots \rightarrow 0 \rightarrow P_i \rightarrow  0 \rightarrow 0 \rightarrow \cdots,\ \ i=1,3,4,
\]
(concentrated  in degree -1) and
\[
\p_2=\cdots \rightarrow 0 \rightarrow P_1\oplus P_3\oplus P_4 \s{p}\rightarrow  P_2 \rightarrow 0 \rightarrow \cdots
\]
where $p$ is a projective presentation of $S_2$. Then it is easily verified that $\p$ is a 2-term silting complex.

By Proposition \ref{prop:summa}(b) we have that  $\T(\p) =\Fac H^0(\p)$,
and hence $\T(\p)=\add S_2.$ We will show the projective dimension of $S_2$ in $\CP$ is infinite, by proving that its third
syzygy equals $S_2$. This implies that a minimal projective resolution
of $S_2$ is periodic and hence infinite.

Using the notation in Section \ref{sec:pre}, we
have that $\widetilde{\p}_1 = \p_1$ and
$\widetilde{\p}_3 = \p_3$. Moreover $\widetilde{\p}_4 = (P_4/S_2)[1]$,
and $\widetilde{\p}_2$ is given by the complex
$$\cdots \rightarrow 0 \rightarrow P_1\oplus P_3\oplus (P_4/S_2) \s{\widetilde{p}}\rightarrow P_2 \rightarrow 0 \rightarrow \cdots$$
Consider now the triangle \[\cone({\pi})[-1]\rightarrow\widetilde{\p}_2\s{\pi}\rightarrow S_2\rightarrow\cone(\bf{\pi})\] where $\pi$ is
\[\begin{array}{ccccccccccc}
\cdots &\rightarrow&0&\rightarrow &P_1\oplus P_3\oplus (P_4/S_2)&\s{\widetilde{p}}\rightarrow&P_2&\rightarrow&0&\rightarrow&\cdots\\
&&\downarrow&&\downarrow&&\pi^0\downarrow&&\downarrow\\
\cdots &\rightarrow &0& \rightarrow &0& \rightarrow  &S_2& \rightarrow &0& \rightarrow &\cdots
\end{array}\]
with $\pi^0$ being a projective cover of $S_2$ in $\mod A$.
%We have that $\widetilde{\p}_2$ is in $\CP$ by Proposition \ref{prop:summa} (d) and Lemma \ref{lem:proj1} (b).

Then $H^0(\cone(\pi)[-1])= 0$, $H^{-1}(\cone(\pi)[-1])\cong H^{-1}(\widetilde{\p}_2)\in\F(\p)$ and $H^i(\cone(\pi)[-1])=0$ for $i\neq-1,0$. So $\cone(\pi)[-1]$ is in $\CP$, using
Proposition \ref{prop:summa} (d).
Hence $\pi$ is a projective cover of $S_2$ in $\CP$ and $\cone(\pi)[-1]$ is its kernel in $\CP$.

Note that $\cone(\pi)[-1]\cong H^{-1}(\cone(\pi)[-1])\cong H^{-1}(\widetilde{\p}_2)\cong
P_1 \oplus M$, where $M=\begin{smallmatrix}\ \ 2\ \ \\1\ 3\ 4\end{smallmatrix}\in\F(\p)$. Consider the triangle \[\cone(\pi_1)[-1]\rightarrow\widetilde{\p}_1\oplus\widetilde{\p}_3\oplus\widetilde{\p}_4\s{\pi_1}\rightarrow M [1]\rightarrow\cone(\pi_1)\] where $\pi_1$ is
\[\begin{array}{ccccccccccc}
\cdots &\rightarrow &0& \rightarrow &P_1\oplus P_3\oplus (P_4/S_2)& \rightarrow  &0& \rightarrow &0& \rightarrow &\cdots\\
&&\downarrow&&\pi_1^{-1}\downarrow&&\downarrow&&\downarrow\\
\cdots &\rightarrow &0& \rightarrow & M & \rightarrow  &0& \rightarrow &0& \rightarrow &\cdots
\end{array}\]
with $\pi_1^{-1}$ being the unique (up to a scalar) right minimal homomorphism from $P_1\oplus P_3\oplus (P_4/S_2)$ to $M$. Then $H^0(\cone(\pi_1)[-1])\cong S_2\in\T(\p)$, $H^{-1}(\cone(\pi_1)[-1])\cong \begin{smallmatrix}2\\1\end{smallmatrix}\oplus M \in\F(\p)$ (since $\Hom_A(S_2,\begin{smallmatrix}2\\1\end{smallmatrix})=0$) and $H^i(\cone(\pi_1)[-1])=0$ for $i\neq-1,0$. So $\cone(\pi_1)[-1]\in\CP$, hence $\pi_1$ is a projective cover of $M[1]$ in $\CP$ and $\cone(\pi_1)[-1]$ is its kernel.

Now consider the triangle \[\cone(\pi_2)[-1]\rightarrow \widetilde{\p}_2\s{\pi_2}\rightarrow\cone(\pi_1)[-1]\rightarrow\cone(\pi_2)\] where $\pi_2$ is
\[\begin{array}{ccccccccccc}
\cdots &\rightarrow &0& \rightarrow &P_1\oplus P_3\oplus (P_4/S_2)& \s{p}\rightarrow  &P_2& \rightarrow &0& \rightarrow &\cdots\\
&&\downarrow&&\pi_2^{-1}\downarrow&&\pi_2^{0}\downarrow&&\downarrow\\
\cdots &\rightarrow &0& \rightarrow &P_1\oplus P_3\oplus (P_4/S_2)& \s{\pi_1^{-1}}\rightarrow  & M & \rightarrow &0& \rightarrow &\cdots
\end{array}\]
with $\pi_2^{-1}$ being the identity map and $\pi^0_2$ being a projective cover of $M$ in $\mod A$. Then we have $H^0(\cone(\pi_2)[-1])\cong S_2\in\T(\p)$ and $H^i(\cone(\pi_1)[-1])=0$ for $i\neq 0$. Hence $\cone(\pi_2)[-1]\cong S_2\in\CP$ and we have a short exact sequence in $\CP$:
\[0\rightarrow S_2\rightarrow \widetilde{\p}_2\s{\pi_2}\rightarrow (\cone(\pi_1)[-1]\rightarrow0.\]
Thus, the projective resolution of $S_2$ in $\CP$ is periodic and hence the
projective dimension is infinite. Therefore, also the global dimension of $B$ is infinite,
by Proposition \ref{prop:summa} (e).

Now, for any $n$ consider the quiver $Q_n$ given by
\[
\xymatrix@C=2pc@R=2pc{
3\ar@<-.5ex>[r]_b & 2\ar@<-.5ex>[l]_a\ar[d]^{c_0}\ar@<-.5ex>[r]_d & 4\ar@<-.5ex>[l]_e &
& & \\
& 1_0 \ar[r]_{c_1} & 1_1  \ar[r]_{c_2} & 1_2  \ar[r]_{c_3} &
\cdots \ar[r]_{c_n} & 1_n
}\]

\noindent with relations $I_n=\langle ba, bd, abc_0, de, c_0 c_1, c_1 c_2, \dots, c_{n-1} c_n\rangle $. Consider the algebra $A(n) = kQ_n/I_n$.
We leave it as an exercise to check that $A(n)$ has global dimension
$n+3$, and to find a 2-term silting complex $\p'$,
such that $\End_{D^b(A(n))}(\p')$ has infinite global dimension.

\end{document}